\documentclass[reqno,final]{amsart}
\usepackage{mathrsfs}
\usepackage{amsthm}
\usepackage{upgreek}
\usepackage{color}
\usepackage{graphicx}
\usepackage{color}
\usepackage{enumerate}

\makeatletter
\def\thm@space@setup{\thm@preskip=0pt
\thm@postskip=0pt}
\makeatother

\renewcommand{\vec}[1]{{\mathchoice
                     {\mbox{\boldmath$\displaystyle{#1}$}}
                     {\mbox{\boldmath$\textstyle{#1}$}}
                     {\mbox{\boldmath$\scriptstyle{#1}$}}
                     {\mbox{\boldmath$\scriptscriptstyle{#1}$}}}}

\newcommand{\eps}{\varepsilon}
\newcommand{\design}{\vec{\xi}}

\newcommand{\LL}{\mathscr{L}}
\newcommand{\Lt}{\mathscr{L}_1}

\newcommand{\Ltp}{\mathscr{L}_1^{\text{sym}+}}

\newcommand{\mat}[1]{\mathbf{{#1}}}
\newcommand{\HMd}{\mat{H}_\text{m}}
\newcommand{\HM}{\mathcal{H}_\text{m}}
\newcommand{\HMt}{\tilde{\mathcal{H}}_\text{m}}
\newcommand{\HMtd}{\tilde{\mat{H}}_\text{m}}
\newcommand{\like}{\pi_\text{like}}
\newcommand{\ZZ}{\mathcal{Z}}
\newcommand{\Zlike}{\ZZ_\text{like}}
\renewcommand{\S}{\mathcal{S}}
\newcommand{\C}{\mathcal{C}}

\newcommand{\A}{\mathcal{A}}
\newcommand{\Q}{\mathcal{Q}}
\newcommand{\R}{\mathbb{R}}

\newcommand{\J}{\mathcal{J}}
\newcommand{\B}{\mathfrak{B}}
\newcommand{\Y}{\mathscr{Y}}
\newcommand{\hilb}{\mathscr{H}}
\newcommand{\trace}{\mathrm{tr}}
\newcommand{\CM}{\hilb_{\priorm}}
\newcommand{\eip}[2]{\displaystyle\left\langle{#1},{#2}\right\rangle_{\R^n}}
\newcommand{\eipq}[2]{\displaystyle\left\langle{#1},{#2}\right\rangle_{\R^q}}
\newcommand{\ip}[2]{\displaystyle\left\langle{#1},{#2}\right\rangle_{\!\hilb}}
\newcommand{\cip}[2]{\displaystyle\left\langle{#1},{#2}\right\rangle_{\!\priorcov^{-1}}}
\newcommand{\cipfd}[2]{\displaystyle\left\langle{#1},{#2}\right\rangle_{\!\mat{C}_\text{pr}^{-1}}}
\newcommand{\avemu}[2]{\mathsf{E}_{#1}\left\{ {#2} \right\}}

\newcommand{\norm}[1]{\left\| {#1} \right\|_\hilb}
\renewcommand{\span}{\mathsf{Span}}

\newcommand{\ran}{\mathsf{range}}
\newcommand{\GM}[2]{\mathcal{N}\!\left({#1},{#2}\right)}
\newcommand{\ipar}{u}
\newcommand{\priorm}{\mu_\text{pr}}
\newcommand{\priorcov}{\C_\text{pr}}
\newcommand{\obs}{\vec{y}}
\newcommand{\postm}{\mu_\text{post}^\obs}
\newcommand{\priormean}{\ipar_\text{pr}}

\newcommand{\postmw}{\mu_\text{post}^{\obs,\design}}
\newcommand{\postcov}{\C_\text{post}}
\newcommand{\postmean}{\ipar_\text{post}^\obs}
\newcommand{\fop}{\mathcal{G}}
\newcommand{\dpar}{\vec{u}}
\newcommand{\dpriorcov}{\mat{C}_\text{pr}}

\newcommand{\dpostcov}{\mat{C}_\text{post}}
\newcommand{\dpostmean}{\dpar_\text{post}^\obs}

\newcommand{\dpriorm}{\mu_\text{pr,n}}
\newcommand{\dpostm}{\mu_\text{post,n}^\obs}
\newcommand{\ncov}{\mat{\Gamma}_\text{noise}}
\newcommand{\DKLtext}[2]{D_\text{kl}({#1} \| {#2})}
\newcommand{\DKL}[2]{D_\text{kl}\left({#1} \| {#2}\right)}

\newcommand{\MSE}{\mathrm{MSE}}

\newcommand{\avey}[1]{\mathsf{E}_{{\obs|\ipar}}\left\{ {#1} \right\}}
\newcommand{\prhalf}{\priorcov^{1/2}}

\newcommand{\Cpv}{\priorcov^{1/2}\S\priorcov^{1/2}}
\newcommand{\dblexp}[1]{\avemu{\priorm}{\avemu{\obs | \ipar}{ {#1}}}}
\newcommand{\Bop}{\mathcal{B}}

\newtheorem{theorem}{Theorem}[]
\newtheorem{proposition}{Proposition}
\newtheorem{lemma}{Lemma}

\theoremstyle{definition}

\theoremstyle{remark}
\newtheorem{remark}{Remark}

\author{Alen Alexanderian}
\address{Institute for Computational Engineering and Sciences, The University of Texas at Austin}
\email{alen@ices.utexas.edu}
\author{Philip Gloor}
\address{Mathematics Department, United States Naval Academy}
\email{gloor@usna.edu}
\author{Omar Ghattas}
\address{Institute for Computational Engineering and Sciences, The University of Texas at Austin}
\email{omar@ices.utexas.edu}


\date{\today}

\title[Bayesian A- and D-optimality in infinite dimensions]{On Bayesian A- and D-optimal experimental designs in infinite dimensions}

\begin{document}

\begin{abstract}
We consider Bayesian linear inverse problems in
infinite-dimensional separable Hilbert spaces, with a Gaussian
prior measure and additive Gaussian noise model,
and provide an extension of the concept of Bayesian D-optimality
to the infinite-dimensional case. To this end,
we derive the infinite-dimensional version of the expression for
the Kullback-Leibler divergence from the posterior measure
to the prior measure, which is subsequently used to
derive the expression for the expected information gain. We
also study the notion of Bayesian A-optimality in the
infinite-dimensional setting, and extend the well known (in the finite-dimensional
case) equivalence of the Bayes risk of the MAP estimator with the trace of the
posterior covariance, for the Gaussian linear case, to the infinite-dimensional Hilbert space case.
\end{abstract}

\maketitle

\section{Introduction}\label{sec:intro}
In a Bayesian inference problem one uses experimental (observed) data to update
the prior state of knowledge about a parameter which   
often specifies certain properties of a mathematical model. 
The ingredients of a Bayesian inference problem include the prior measure which
encodes our prior knowledge about the inference parameter, experimental data, and the data 
likelihood which describes the conditional distribution of the experimental data for a given model parameter.  
The solution of a Bayesian inference problem is 
a posterior probability law for the inference parameter. The quality of this solution, which 
can be measured using different criteria, depends to a large extent on the experimental data used
in solving the inference problem. In practice, acquisition of such experimental data is often costly, 
as it requires deployment of scarce resources. Hence, the problem of optimal collection
of experimental data, i.e.~that of optimal experimental design (OED)~\cite{AtkinsonDonev92,Ucinski05,Pukelsheim06}, is an integral 
part of modeling under uncertainty. 
The basic problem of OED is to optimize a function of the experimental 
setup which describes, in a certain sense which needs to be specified, the statistical quality of the 
solution to the Bayesian inference problem. 
Note that what constitutes an experimental design depends on the application at hand. 
For example, in a problem involving diffusive transport of a contaminant, one may 
use measurements of concentration at sensor sites in the physical domain (at a certain point in time) 
to \emph{infer} where the contaminant originated, i.e. the initial state of the concentration field.
In this problem, an experimental design specifies the locations of the sensors in the
physical domain. Note also that the inference parameter in this example, i.e.~the initial concentration field, 
is a random function (random field) whose realizations belong to an appropriate function space.

We consider the problem of design of experiments for inference problems  
whose inference parameter belongs to an infinite-dimensional separable Hilbert space. 
This is motivated by the recent interest in the Bayesian framework for inverse 
problems~\cite{Stuart10}.
A \emph{Bayesian inverse problem} involves inference of Hilbert space valued
parameters that describe physical properties of mathematical models which are often
governed by partial differential equations. Study of such problems requires a synthesis of 
ideas from inverse problem theory, PDE-constrained optimization, functional analysis, 
and probability and statistics and has provided a host of interesting 
mathematical problems with a wide range of applications. 
The problem of design of experiments in this 
infinite-dimensional setting involves optimizing functionals of experimental designs 
which are defined in terms of operators on Hilbert spaces. 

The precise definition of what is meant by an \emph{optimal} design 
leads to the choice of a design criterion.
A popular experimental design criterion, in the finite-dimensional case, is  
that of D-optimality 
which seeks to minimize the determinant of the posterior covariance operator.
The geometric intuition behind D-optimality is that of minimizing the \emph{volume} of the uncertainty ellipsoid. 
Minimizing this determinant, however, is not meaningful in infinite dimensions, 
as the posterior covariance operator is a trace-class linear operator 
with eigenvalues that accumulate at zero. In the present work, we provide an extension 
of the concept of D-optimal design to the infinite-dimensional Hilbert space setting. 
In particular, we focus on the case of Bayesian linear inverse problems whose
parameter space is an infinite-dimensional separable Hilbert space which we denote by $\hilb$, 
and we assume a Gaussian prior measure, and an additive Gaussian noise model. 
To study the concept of D-optimality in the infinite-dimensional setting
we formulate the problem as that of maximizing the expected
information gain, measured by the Kullback-Leibler (KL) divergence~\cite{Kullback1951} 
from posterior to prior. To be precise, if $\priorm$ denotes the prior measure, 
$\obs$ is a vector of experimental data obtained using an experimental design specified by 
a vector of design parameters $\design$, 
and $\postmw$ denotes the resulting posterior measure, the KL divergence from posterior to 
prior is given by, 
\[
    \DKL{\postmw}{\priorm} := \int_\hilb \log\left\{ \frac{d\postmw}{d\priorm} \right\}\, d\postmw.
\]
(The argument of the logarithm in the above formula is the Radon-Nikodym derivative of the posterior 
measure with respect to the prior measure.) 
The experimental design criterion is then defined by averaging $\DKLtext{\postmw}{\priorm}$ over all possible
experimental data.
In a Bayesian inverse problem, this averaging over experimental data can be done as follows: 
\[
    \text{expected information gain} :=  \int_\hilb \int_\Y \DKL{\postmw}{\priorm}\, \like(\obs | \ipar; \design)d\obs \, \priorm(d\ipar),
\]
where $\design$ is a fixed design vector, $\Y$ denotes the space of experimental data and $\like(\obs | \ipar; \design)$ is the data likelihood which 
specifies the distribution of $\obs$ for a given $\ipar \in \hilb$.

It is known in the finite-dimensional Gaussian linear case (i.e., an inference problem
with Gaussian prior and noise distributions) that 
maximizing this expected information gain is equivalent to minimizing 
the determinant of the posterior covariance operator, i.e., the usual D-optimal design problem. 
While this does not directly extend to 
the infinite-dimensional case, it suggests a mathematically rigorous path to an 
infinite-dimensional analogue of Bayesian D-optimality. 
In the present work, we derive analytic expressions for the KL divergence from posterior to prior in a Hilbert space.
This enables deriving the expression for the expected information gain, leading
to the infinite-dimensional version of the Bayesian D-optimal experimental design criterion. 

We also discuss another popular experimental design criterion, that of A-optimality, in 
the infinite-dimensional setting. An A-optimal design is one that 
minimizes the trace of the posterior covariance operator; i.e., if 
$\postcov(\design):\hilb \to \hilb$ denotes the posterior covariance operator 
corresponding to an experimental design $\design$, we seek to minimize $\trace(\postcov(\design))$.
In the statistics literature it is known (see e.g.,~\cite{ChalonerVerdinelli95}) that 
for a Gaussian linear inference problem in $\hilb = \R^n$, minimizing the 
trace of the posterior covariance \emph{matrix} is equivalent to 
minimizing the average mean square error of the maximum a 
posteriori probability (MAP) estimator for the inference parameter. 
We provide an extension of this result to the infinite-dimensional Hilbert space setting, where 
we show that the trace of the posterior covariance \emph{operator}---a positive, self-adjoint, and trace-class 
operator on $\hilb$---coincides with the average mean square error of the MAP estimator.  

Note that the design vector $\design$ enters the Bayesian inverse problem through the data likelihood. 
The exact nature of this dependence on $\design$ is not essential to our discussion and hence, to keep the
notation simple, we suppress the dependence to $\design$ in our 
derivations. (See e.g.,~\cite{ChalonerVerdinelli95} for a 
an overview of how an experimental design is incorporated in an inference problem 
in classical formulations.)

\section{Background concepts}\label{sec:background}
In this section, we outline the background concepts that are needed 
in the rest of this article. 
In what follows, $\hilb$ denotes an infinite-dimensional separable
real Hilbert space, with inner-product $\ip{\cdot}{\cdot}$ and induced
norm $\norm{\cdot} = \ip{\cdot}{\cdot}^{1/2}$.

\subsection{Trace-class operators on $\hilb$}\label{sec:LinearOperators}
Let $\LL(\hilb)$ denote the set of bounded linear operators on $\hilb$.
We say $\A \in \LL(\hilb)$ is positive if $\ip{x}{\A x} \geq 0$ for all
$x \in \hilb$, and is strictly positive if $\ip{x}{\A x} > 0$ for all non-zero $x \in \hilb$.
For $\A \in \LL(\hilb)$, $|\A| = (\A^*\A)^{1/2}$, where $\A^*$ denotes the
adjoint of $\A$. We say $\A$ is of \emph{trace-class} if for any orthonormal basis $\{ f_j\}_{j=1}^\infty$ of
$\hilb$,
\[
\sum_{j = 1}^\infty \ip{|\A|f_j}{f_j} < \infty.
\]
It is straightforward to show that the value of the above summation is invariant with respect to the choice
of the orthonormal basis~\cite{ReedSimon}. We denote by
$\Lt(\hilb)$ the subspace of $\LL(\hilb)$ consisting of trace-class operators.
For $\A \in \Lt(\hilb)$,
\[
    \trace(\A) = \sum_{j = 1}^\infty \ip{\A f_j}{f_j},
\]
where the sum is finite and its value is independent of the choice of the
orthonormal basis~\cite{Conway,ReedSimon}. 

Let $\Ltp(\hilb)$ be the subspace of positive self-adjoint operators in $\Lt(\hilb)$,
and note that for $\A \in \Ltp(\hilb)$, there exists an
orthonormal basis of eigenvectors, $\{e_j\}$, with
corresponding (real, non-negative) eigenvalues, $\{\lambda_j\}$, and 
$\trace(A) = {\sum_{j = 1}^\infty \ip{\A e_j}{e_j}} = \sum_{j = 1}^\infty
   \lambda_j$.

In what follows we shall make repeated use of the following result: if $\A \in \Lt(\hilb)$
and $\Bop \in \LL(\hilb)$ then $\A \Bop$ and $\Bop \A$ both belong to $\Lt(\hilb)$ and
$\trace(\A \Bop) = \trace(\Bop \A)$; see e.g.,~\cite{ReedSimon} for details. 
Moreover, it is straightforward to show that if 
$\A$ is a trace-class operator and $\Bop:\hilb \to \R^q$ is a bounded linear operator, 
then $\A\Bop^*\Bop \in \Lt(\hilb)$ and $\trace(\A\Bop^*\Bop) = \trace(\Bop\A\Bop^*)$.

\subsection{Borel probability measures on $\hilb$}\label{sec:BorelMeas}
We work with probability measures on the measurable space $\big(\hilb, \B(\hilb)\big)$, 
where $\B(\hilb)$ denotes the Borel sigma-algebra on $\hilb$; we refer to such 
measures as Borel probability measures. 
Let $\mu$ be a Borel probability measure on $\hilb$, which has bounded first
and second moments. The mean $m \in \hilb$ and covariance operator $\Q \in \LL(\hilb)$
of $\mu$ are characterized as follows:
\[
   \ip{m}{x} = \int_\hilb \ip{z}{x} \, \mu(dz), \qquad
    \ip{\Q x}{y} = \int_\hilb \ip{x}{z - m}\ip{y}{z - m}\, \mu(dz),
\]
for all $x, y \in \hilb$.
It is straightforward to show (see e.g.,~\cite{DaPrato}) that
$\Q$ belongs to $\Ltp(\hilb)$, and that
\begin{equation}\label{equ:second_moment}
    \int_\hilb \norm{x}^2 \, \mu(dx) = \trace(\Q) + \norm{m}^2.
\end{equation}

\subsection{Gaussian measures on $\hilb$}\label{sec:GaussianMeas}
In the present work, we shall be working with Gaussian measures on
Hilbert spaces~\cite{DaPrato}; $\mu$ is a Gaussian measure
on $(\hilb, \B(\hilb))$ if for every $x \in \hilb$ the linear functional
$\ip{x}{\cdot}$, considered as a random variable from $(\hilb, \B(\hilb), \mu)$ to $(\R, \B(\R))$,
is a (one-dimensional)  Gaussian random variable. We refer the
reader to~\cite{DaPrato} or \cite{DaPratoZabczyk} for the theory of Gaussian measures
on Hilbert spaces. We denote a Gaussian measure with mean 
$m \in \hilb$ and $\Q \in \Ltp(\hilb)$ by $\GM{m}{\Q}$. If $\Q$ satisfies 
$\ker(\Q) = \{ 0 \}$, where $\ker(\Q)$ denotes the null space of $\Q$, we say that $\GM{m}{\Q}$ is
a non-degenerate Gaussian measure.  

In what follows, we shall use the following result, concerning the law
of an affine transformation on $\hilb$: If $\mu = \GM{m}{\Q}$, a Gaussian measure, $\A \in \LL(\hilb)$, and $b \in \hilb$, then $Tx = \A x+b$ is a random
variable on $\hilb$ whose law is given by
$\mu_T = \mu \circ T^{-1} = \GM{\A m + b}{\A\Q\A^*}$~\cite{DaPrato}.
Thus, in particular, we note that,
\[
   \int_\hilb \norm{T x}^2 \, \mu(dx) = \int_\hilb \norm{\xi}^2 \, \mu_T(d\xi)
   = \trace(\A \Q \A^*) + \norm{\A m + b}^2,
\]
where the last equality uses~\eqref{equ:second_moment}.
It follows that if $\A \in \LL(\hilb)$ is positive self-adjoint compact operator, and 
$\mu = \GM{m}{\Q}$ is a Gaussian measure, then

\begin{eqnarray}
   \int_{\hilb} \ip{\A x}{x} \, \mu(dx) &=& \int_{\hilb} \norm{\A^{1/2} x}^2 \, \mu(dx)  \nonumber\\ 
   &=& \trace(\A^{1/2} \Q \A^{1/2}) + \ip{\A^{1/2} m}{\A^{1/2}m} \\
   &=& \trace(\A\Q) + \ip{\A m}{m}.\nonumber
\end{eqnarray}
This shows that the well-known expression for the expectation of a quadratic form on $\R^n$ 
extends to the infinite-dimensional Hilbert space setting.
It can be shown that, as in the finite-dimensional case, 
this result holds not just for Gaussian measures, but also for any Borel 
probability measure with mean $m$ and covariance operator $\Q$; moreover, 
the the only requirement on the operator $\A$ is boundedness. That is, we have the 
following result:
\begin{lemma}\label{lem:quadforms}
Let $\mu$ be a Borel probability measure on $\hilb$
with mean $m \in \hilb$ and covariance operator $\Q \in \Ltp(\hilb)$,
and let $\A \in \LL(\hilb)$. Then,
\[
   \int_{\hilb} \ip{\A x}{x} \, \mu(dx) = \trace(\A\Q) + \ip{\A m}{m}.
\]
\end{lemma}
\begin{proof}
See Appendix~\ref{appdx:quadforms}.\qedhere
\end{proof}
\noindent

\subsection{Kullback-Leibler divergence}\label{sec:KLDiv}
In probability theory the Kullback-Leibler (KL) divergence, also referred to as the relative entropy, is a 
measure of ``distance'' between two probability measures. This notion was defined
in~\cite{Kullback1951}. While KL divergence is not a metric---it is non-symmetric 
and does not satisfy the triangle inequality---it is used commonly in 
probability theory to describe the distance of a measure $\mu$ from a reference measure $\mu_0$. 
Also, KL divergence does satisfy some of the intuitive notions of distance; i.e. the KL
divergence from $\mu$ to $\mu_0$ is non-negative and is zero if and only if the two measures are the same. 
Consider $\mu$ and $\mu_0$ be two Borel probability measures and suppose $\mu$ is absolutely continuous
with respect to $\mu_0$. The KL divergence from $\mu$ to $\mu_0$, denoted by
$\DKL{\mu}{\mu_0}$, is defined as
\[
    \DKL{\mu}{\mu_0} = \int_\hilb \log \Big\{\frac{d \mu}{d \mu_0}\Big\} \, d\mu,
\]
where $\frac{d\mu}{d\mu_0}$ is the Radon-Nikodym derivative of $\mu$ with respect to $\mu_0$.
In the case $\mu$ is not absolutely continuous with respect to $\mu_0$ the KL divergence is $+\infty$. 
Notice that for Borel probability measures on $\R^n$ that admit densities with 
respect to the Lebesgue measure, we may rewrite the definition of the KL 
divergence in terms the densities; that is, if $p$ and $p_0$ are Lebesgue densities, i.e., probability density functions (pdfs), 
of $\mu$ and $\mu_0$ respectively, one has $\DKL{\mu}{\mu_0} = \int_{\R^n} \log\big(p(\vec{x})/p_0(\vec{x})\big)\, p(\vec{x}) \, d\vec{x}$.
However, in an infinite-dimensional Hilbert space, where there is no Lebesgue measure, we are forced to work
with the abstract definition of KL divergence presented above. 

In this paper, we will be dealing with (non-degenerate) Gaussian measures on infinite-dimensional Hilbert spaces. 
For Gaussian measures on $\R^n$, one can use the expression for the (multivariate) 
Gaussian pdfs to derive the well-known analytic expression for the KL divergence between Gaussians. In the 
infinite-dimensional Hilbert space setting, not only do we not have access to pdfs, but 
given two Gaussian measures they are not necessarily equivalent.\footnote{Recall that two measures are called 
equivalent if they are mutually absolutely continuous with respect to each other.} 
In fact, given a centered Gaussian measure $\mu$ on $\hilb$, shifting the mean gives, $\mu$-almost surely, a 
Gaussian measure which is singular with respect to $\mu$; see e.g.,~\cite[Chapter 2]{DaPrato}. 
However, 
In the present work, we work with a special case, namely that of 
a Bayesian linear inverse problem on $\hilb$ with a Gaussian prior and
an additive Gaussian noise model; in this case the posterior measure 
is also Gaussian and is equivalent to the prior~\cite{Stuart10}, and thus, $\DKL{\postm}{\priorm}$ is well-defined.
Later in the paper, we will derive the expression for the KL divergence 
from posterior to prior in an infinite-dimensional Hilbert space, which we shall use to derive the expression for 
the expected information gain. 

\section{Bayesian linear inverse problems in a Hilbert space}\label{sec:BayesInverse}
We consider the problem of inference of a parameter $\ipar$ which belongs to an infinite-dimensional Hilbert
space $\hilb$. 
All our prior knowledge regarding the parameter $\ipar$
is encoded in a Borel probably measure on $\hilb$, which we 
refer to as the prior measure and denote by $\priorm$;
here we assume that $\priorm$ is a Gaussian measure 
$\priorm = \GM{\priormean}{\priorcov}$. Moreover, in what follows, 
we assume that $\ker(\priorcov) = \{0\}$, i.e., $\priorm$ is non-degenerate.
The inference problem uses experimental data $\obs \in \Y$ to update the prior state of
knowledge on the law of the parameter $\ipar$. 
Here $\Y$ is the space of the experimental data, which in the present work is $\Y = \R^q$. 
We assume that $\ipar$ is a model parameter 
which is related to experimental data $\obs \in \Y$ 
according to the following noise model,
\begin{equation}\label{equ:noise}
    \obs = \fop \ipar + \vec{\eta}. 
\end{equation}
The operator $\fop:\hilb \to \Y$ is the \emph{parameter-to-observable map} and 
is assumed to be a continuous linear mapping.  In practice, for a given $\ipar$, computing $\fop \ipar$ would involve the 
evaluation of a mathematical model with the parameter value $\ipar$ 
followed by the application of a restriction operator to extract data at pre-specified locations
in space and/or time. 
The discrepancy between the model output $\fop u$ and experimental data $\obs$ is modeled by  
$\vec{\eta}$ which is a random vector that accounts for experimental noise, i.e. noise associated
with the process of collecting experimental data. We assume $\vec{\eta} \sim \GM{\vec{0}}{\ncov}$, and 
thus, the distribution of $\obs | \ipar$ is Gaussian, $\obs | \ipar \sim \GM{\fop \ipar}{\ncov}$ with pdf 
\[
   \like(\obs | \ipar) = \frac{1}{\Zlike} \exp \Big\{ -\frac12 (\fop\ipar -
  \obs)^T \ncov^{-1} (\fop\ipar - \obs) \Big\},
\]
where $\Zlike = (2\pi)^{q/2}\det(\ncov)^{1/2}$.
In what follows, we denote
\begin{equation}\label{equ:phi}
\Phi(\ipar; \obs) = \frac12 (\fop\ipar - \obs)^T \ncov^{-1} (\fop\ipar - \obs). 
\end{equation}

\subsection{The Bayes formula and the posterior measure}
The solution of the Bayesian inverse problem is the posterior measure, describing the law of the parameter $\ipar$,
conditioned on the experimental data $\obs$, and is linked to the prior measure $\priorm$ through
the infinite-dimensional version of Bayes Theorem~\cite{Stuart10}:
\begin{equation}\label{equ:bayes}
   \frac{d\postm}{d\priorm} = \frac{1}{\ZZ(\obs)}\like(\obs|\ipar),
\end{equation}
where $\ZZ(\obs)$ is the normalization constant. Notice that we can rewrite Bayes Theorem as,
\begin{equation}\label{equ:bayes-phi}
   \frac{d\postm}{d\priorm} = \frac{1}{\ZZ_0(\obs)} \exp\{-\Phi(u; \obs)\},
\end{equation}
with $\ZZ_0(\obs) = 
\int_\hilb \exp\{-\Phi(u; \obs)\}\, \priorm(d\ipar)$. In the Gaussian linear case, 
it is possible to evaluate $\ZZ_0(\obs)$ analytically; see Lemma~\ref{lem:Zconst} below. 

As discussed above, we consider Bayesian linear inverse problems; i.e.,  
Bayesian inverse problems involving a linear parameter-to-observable map $\fop$.
It is well known~\cite{Stuart10} that for a Gaussian linear inverse problem, as specified above,
the solution is a Gaussian posterior measure $\postm = \GM{\postmean}{\postcov}$ with,
\[
\postcov = (\fop^* \ncov^{-1} \fop + \priorcov^{-1})^{-1},
\qquad
\postmean = \postcov(\fop^* \ncov^{-1} \obs + \priorcov^{-1}\priormean).
\]
In practice, the noise covariance matrix, $\ncov$ is often a multiple of the identity, $\ncov = \sigma^2 \mat{I}$,
where $\sigma$ is the experimental noise level. 
In the derivations that follow, since there is no loss of generality, we take $\sigma = 1$. 
Generalizing the results to the cases where $\ncov$ is an anisotropic diagonal matrix (uncorrelated observations with  
varying experimental noise levels) or more generally $\ncov$ that is symmetric and positive 
definite with nonzero off diagonal entries (correlated observations) is straightforward. 
Moreover, for simplicity, we assume that the prior is a centered Gaussian, i.e.  $\priormean = 0$.
Again, the generalization to the case of non-centered prior measure is straightforward.
With these simplifications, the mean and covariance of the posterior measure are given by,
\begin{equation}\label{equ:postmeas}
\postcov = (\fop^*\fop + \priorcov^{-1})^{-1},
\qquad
\postmean = \postcov\fop^*\obs.
\end{equation}
In what follows, we use the notation,
\begin{equation}\label{equ:HM}
\HM = \fop^*\fop. 
\end{equation}
The motivation behind this notation is 
that $\fop^*\fop$ is the Hessian of the functional, $\Phi(\ipar; \obs)$, which measures the
magnitude of the \emph{misfit} between experimental data $\obs$ and model prediction $\fop \ipar$.
Note that in statistical terms, $\HM$ is the Hessian of the negative log-likelihood which is also 
referred to as the \emph{Fisher information matrix}.
Another notation we shall use frequently is,
\begin{equation}\label{equ:HMt}
   \HMt = \priorcov^{1/2} \HM \priorcov^{1/2}.
\end{equation}
Intuitively, this \emph{prior-preconditioned} $\HM$ can be thought of as the information matrix
which has been filtered through the prior. 
To further appreciate the notion of the prior-preconditioned misfit Hessian,
we note that the second moment of the parameter-to-observable map, considered as a 
random variable $\fop:(\hilb, \B(\hilb), \priorm) \to (\R^q, \B(\R^q))$ is given by,
\[
\int_\hilb |\fop u|^2 \,\priorm(du) = \int_\hilb \eipq{\fop u}{\fop u}\,\priorm(du)  = \int_\hilb \ip{\HM u}{u}\,\priorm(du)  
= \trace(\priorcov \HM) = \trace(\HMt). 
\]

\subsection{A spectral point of view of uncertainty reduction}
Let $\HMt$ be the prior-preconditioned misfit Hessian as defined in~\eqref{equ:HMt} and denote
\begin{equation}\label{equ:Sop}
\S = (I + \HMt)^{-1}.
\end{equation}
The posterior covariance operator, $\postcov$, given in~\eqref{equ:postmeas} 
can be written as, $\postcov = \prhalf (I + \HMt)^{-1} \prhalf = \prhalf \S \prhalf$.
We consider the quantity,
\[
   \delta(\priorcov, \postcov) := \trace(\priorcov) - \trace(\postcov) = \trace\big(\prhalf(I - \S)\prhalf\big).
\]
For the class of Bayesian linear inverse problems considered in the present work, it is
straightforward to show that $\delta(\priorcov,\postcov) \geq 0$. In particular, we
note that if $\{\lambda_i\}$ and $\{e_i\}$ are the eigenvalues and the respective eigenvectors of
$\HMt$, then
\[
\ip{e_i}{(I - \S)e_i} = 1 - \ip{e_i}{\S e_i} = 1 - 1/(1+\lambda_i) = \lambda_i/(1+\lambda_i) \geq 0,
\quad i = 1, 2, \ldots,
\]
which shows that $\delta(\priorcov,\postcov) = \trace(\prhalf(I - \S)\prhalf) \geq 0$.
The quantity $\delta(\priorcov,\postcov)$ can thus be considered a measure of
variance (uncertainty) reduction. More precisely,
we consider for each $i \geq 1$,
\[
   \ip{e_i}{\postcov e_i} = \int_\hilb \ip{e_i}{u - \postmean}^2\postm(du),
\]
which measures the posterior variance of the coordinate of $u$ in the direction $e_i$.
\begin{proposition}
Let $\{\lambda_i, e_i\}_1^\infty$ be eigenpairs of $\HMt$. Then, $\ip{e_i}{\postcov e_i} \leq \ip{e_i}{\priorcov e_i}$,
for all $i \geq 1$.
\end{proposition}
\begin{proof}
Note that for each $v \in \hilb$, $\S v = \sum_j (1+\lambda_j)^{-1} \ip{e_j}{v} e_j$.
Hence,
\begin{multline*}
   \ip{e_i}{\postcov e_i} = \ip{e_i}{\prhalf \S \prhalf e_i}
      = \ip{\prhalf e_i} {\S \prhalf e_i} \\
      =\sum_j (1+\lambda_j)^{-1} \ip{e_j}{\prhalf e_i}^2
   \leq \sum_j  \ip{e_j}{\prhalf e_i}^2 = \norm{ \prhalf e_i}^2 = \ip{e_i}{\priorcov e_i},
\end{multline*}
where the penultimate equality follows from Parseval's identity.\qedhere
\end{proof}
Also,
\[
 \trace(\postcov) = \trace(\priorcov) - \trace(\prhalf(I - \S)\prhalf)
                  = \sum_{j = 1}^\infty (1 - \alpha_j) \ip{e_j}{\priorcov e_j},
\]
where $\alpha_j = \lambda_j / (1 + \lambda_j)$. Thus, for eigenvalues $\lambda_j$ that are large, we have $\alpha_j \approx 1$ which
suggests that significant uncertainty reduction occurs in such directions.
It is well known that for large classes of
ill-posed Bayesian inverse problems, the eigenvalues $\lambda_i$ of $\HMt$ decay rapidly to zero, with a relatively
small number of dominant eigenvalues indicating the \emph{data-informed} directions in the parameter space.
This allows ``focusing'' the inference to low-dimensional subspaces of the parameter space $\hilb$. Such
ideas have been used to develop efficient numerical algorithms
for solution of infinite-dimensional Bayesian inverse problems in works such as~\cite{BuiEtAl2013,FlathEtAl2011} 
and for algorithms for computing A-optimal experimental designs for infinite-dimensional Bayesian linear
inverse problems in~\cite{AlexanderianEtAl2014}.

\section{KL divergence from posterior to prior and expected information gain}\label{sec:kldiv}

Let us first motivate the discussion by recalling the form of the KL divergence from the
posterior to prior in the finite-dimensional case. We use boldface letters for the finite-dimensional
versions of the operators appearing in the Bayesian inverse problem. 
To indicate that we work in $\R^n$, we denote by $\dpriorm$ and $\dpostm$ the prior and posterior
measures in the $n$-dimensional case. The following expression for $\DKL{\dpostm}{\dpriorm}$ is
well known:
\begin{multline}\label{equ:DKL-fd}
  \DKL{\dpostm}{\dpriorm}\\ = \frac{1}{2}\,\Big[ -\log\left( \frac{\det\dpostcov}{\det \dpriorcov}\right) - n + \trace(\dpriorcov^{-1} \dpostcov) + \eip{\dpriorcov^{-1}\dpostmean}{\dpostmean} 
                      \Big].
\end{multline}
Note that the above expression is not meaningful in the infinite-dimensional case. 
For one thing, $n$ appears explicitly in the expression. 
Moreover, in the infinite-dimensional case, $\priorcov$ is a trace-class operator
whose eigenvalues accumulate at zero, so dividing by the determinant of the prior covariance is problematic 
as $n \to \infty$. Finally, in the infinite-dimensional case, $\priorcov^{-1}$ is
the inverse of a compact operator and hence is unbounded; therefore, the trace term, which 
involves the inverse of the prior covariance, needs clarification. However, if we reformulate the above expression, we obtain an expression that has meaning in the infinite-dimensional case. 

A straightforward calculation shows that the first term on the right in~\eqref{equ:DKL-fd} may be simplified: 
\begin{eqnarray}\label{equ:detterm}
  -\log\left( \frac{\det\dpostcov}{\det \dpriorcov}\right) \!\!\!&=&\!\!\! \log\left( \frac{\det\dpriorcov}{\det \dpostcov}\right) = \log \det \left( \dpriorcov \dpostcov^{-1}\right) \nonumber \\
  &&\hspace{-.75in}= \log \det \left( \dpriorcov^{1/2} (\HMd + \dpriorcov^{-1}) \dpriorcov^{1/2} \right) \\
  &&\hspace{-.25in}= \log \det (\HMtd + \mat{I}). \nonumber
\end{eqnarray}

\noindent Recall that, in general, if $\mathbf{A}$ is Hermitian, then there exists a unitary matrix $\mathbf{U}$ such that 
\begin{equation*}
  \mathbf{D} = [\lambda_i\delta_{ij}] = \mathbf{U^\ast\!AU}
\end{equation*}
is diagonal. In this case, the diagonal elements are the eigenvalues of $\mathbf{A}$, and
\begin{equation*}
  \det(\mathbf{I} + \mathbf{A}) = \det(\mathbf{U})\det(\mathbf{I} + \mathbf{A})\det(\mathbf{U^\ast}) = \prod_{i=1}^n(1 + \lambda_i).
\end{equation*}
In the infinite-dimensional setting, given a trace-class operator $\A \in \Ltp(\hilb)$, 
\begin{equation*}
  \lim_{n\rightarrow\infty}\log\left(\prod_{i=1}^n(1 + \lambda_i(\A))\right) = \lim_{n\rightarrow\infty}\sum_{i=1}^n\log(1 + \lambda_i(\A)) \leq \lim_{n\rightarrow\infty}\sum_{i=1}^n\lambda_i(\A) < \infty,
\end{equation*}
so, motivated by the $n$-dimensional case, we may define the Fredholm determinant of $I + \A$ as
\[
   \det (I + \A) = \prod_{i = 1}^\infty (1 + \lambda_i(\A)), 
\]
where $\lambda_i(\A)$ are the eigenvalues of $\A$~\cite{Simon1977}. Hence, the final expression in equation~\eqref{equ:detterm} is meaningful in infinite dimensions.
Next, we consider the term $-n+\trace(\dpriorcov^{-1} \dpostcov)$:
\begin{eqnarray}\label{equ:traceterm}
  &&-n+\trace(\dpriorcov^{-1} \dpostcov) = -\trace(\mat{I}) + \trace(\dpriorcov^{-1} \dpostcov) \nonumber \\
  &&\hspace{.75in}= \trace( \dpriorcov^{-1} \dpostcov - \mat{I}) = \trace\big( (\dpriorcov^{-1} - \dpostcov^{-1})\dpostcov\big) 
= -\trace(\HMd \dpostcov),\nonumber
\end{eqnarray}

\noindent where in the last step we used the fact that $\dpostcov^{-1} = \HMd + \dpriorcov^{-1}$.
Notice that the argument of the trace in the final expression is in fact a trace-class operator
in the infinite-dimensional case and has a well-defined trace.
Combining~\eqref{equ:detterm} and~\eqref{equ:traceterm}
and defining the
inner-product $\cipfd{\vec{x}}{\vec{y}} = \eip{\dpriorcov^{-1/2}\vec{x}}{\dpriorcov^{-1/2}\vec{y}}$ for $\vec{x}, \vec{y} \in \R^n$,
we rewrite~\eqref{equ:DKL-fd},  
\begin{equation}\label{equ:DKL-fd-alt}
  \DKL{\dpostm}{\dpriorm} =  
  \frac{1}{2}\Big[ \log\det (\HMtd + \mat{I}) - \trace(\HMd \dpostcov) + 
             \cipfd{\dpostmean}{\dpostmean} 
             \Big].
\end{equation}

In Section~\ref{sec:DKLbayes} we derive, rigorously, alternate forms of the expression for the 
KL divergence from posterior to prior in the infinite-dimensional 
Hilbert space setting; as we shall see shortly, one of those forms is a direct extension 
of~\eqref{equ:DKL-fd-alt} to the infinite-dimensional case.
The reason for introducing the weighted inner-product $\cipfd{\cdot}{\cdot}$ will also become clear 
in the discussion that follows.

\subsection{The KL-divergence from posterior to prior}\label{sec:DKLbayes}
The following result which is a consequence of Proposition 1.2.8 in~\cite{DaPratoZabczyk} will be needed in
what follows.
\begin{proposition}\label{prp:exp_quad_form}
Let $\A \in \LL(\hilb)$ be a positive self-adjoint operator, $\mu = \GM{0}{\Q}$ a non-degenerate 
Gaussian measure on $\hilb$, and
$b \in \hilb$.
Then,
\newcommand{\AT}{\tilde{\mathcal{A}}}
\[
\int_\hilb \exp\left\{-\frac12 \ip{\A x }{x} + \ip{b}{x}\right\} \, \mu(dx) = 
\det (I + \AT)^{-1/2} \exp\Big\{ \frac12 \big\|(I + \AT)^{-1/2} \Q^{1/2} b\big\|_\hilb^2\Big\},
\]
where $\AT = \Q^{1/2}\A \Q^{1/2}$.
\end{proposition}

In the following technical lemma, we calculate the expression for $\ZZ_0$, introduced in equation~\eqref{equ:bayes-phi}.

\begin{lemma}\label{lem:Zconst}
Let $\Phi(\ipar; \obs) = \frac12 (\fop\ipar - \obs)^T \ncov^{-1} (\fop\ipar - \obs)$, as defined by equation~\eqref{equ:phi}. Then,
\[
  \ZZ_0(\obs) := \int_\hilb \exp\{-\Phi(u; \obs)\} \, \priorm(du) = \exp\left\{-\frac12 |\obs|^2\right\} \det(I + \HMt)^{-1/2}  \exp\left\{\frac12 \ip{\postcov b}{b}\right\},
\]
where $b = \fop^*\obs$ and $\postcov = (\fop^*\fop + \priorcov^{-1})^{-1}$, as in equation~\eqref{equ:postmeas}.
\end{lemma}

\begin{proof}
First note that (recall that we have assumed $\ncov = \mat{I}$)
\begin{eqnarray}\label{equ:phi-simp}
   \Phi(\ipar; \obs) \!\!\!&=&\!\!\! \frac12(\fop \ipar - \obs)^T (\fop \ipar - \obs) = \frac12 \eipq{\fop \ipar}{\fop \ipar} - \eipq{\fop \ipar}{\obs} +\frac12 \eipq{\obs}{\obs} \nonumber \\
   &&\hspace{-.5in}
= \frac12 \ip{\HM \ipar}{\ipar} - \ip{\fop^*\obs}{\ipar} + \frac12 |\obs|^2.
\end{eqnarray}
Therefore,
\[
\int_\hilb \exp\{-\Phi(u; \obs)\} \, \priorm(du) = \exp\left\{-\frac12 |\obs|^2\right\} \int_\hilb   
                                    \exp\left\{-\frac12 \ip{\HM \ipar}{\ipar} + \ip{b}{\ipar} \right\}  \priorm(d\ipar),
\]
where $b = \fop^*\obs$.
By Proposition~\ref{prp:exp_quad_form} we have,
\begin{eqnarray*}
  &&\int_\hilb \exp\left\{-\frac12 \ip{\HM \ipar}{\ipar} + \ip{b}{\ipar} \right\}  \priorm(d\ipar) 
\\
  &&\hspace{.5in}= \det(I + \priorcov^{1/2} \HM \priorcov^{1/2})^{-1/2} 
  \exp\left\{\frac12 \norm{(I + \priorcov^{1/2} \HM \priorcov^{1/2})^{-1/2}\priorcov^{1/2}b}^2\right\}
\\
  &&\hspace{1in}= \det(I + \HMt)^{-1/2}
  \exp\left\{\frac12 \norm{(I + \HMt)^{-1/2}\priorcov^{1/2}b}^2\right\}.
\end{eqnarray*}
The assertion of the lemma now follows, since
$\postcov = 
\priorcov^{1/2}(I + \HMt)^{-1} \priorcov^{1/2}$.\qedhere
\end{proof}

The following result provides the expression for the KL divergence from posterior to prior:

\begin{proposition}\label{prp:dkl}
Let $\priorm$ be a centered Gaussian measure on $\hilb$, and $\postm = \GM{\postmean}{\postcov}$ be
the posterior measure for a Bayesian linear inverse problem with additive Gaussian noise model
as described in Section~\ref{sec:BayesInverse}. Then,
\begin{equation}\label{equ:DKLbayes}
    \!\!
    \DKL{\postm}{\priorm} \!=\! \frac12 \left[\log \det(I + \HMt) - \trace(\HM \postcov) - 
    \ip{\postmean}{\fop^*(\fop\postmean - \obs)}\right].\!\!\!
\end{equation}
\end{proposition}

\begin{proof} Consider~\eqref{equ:bayes-phi},
and note that
\begin{equation}\label{equ:dkl}
\begin{aligned}
    \DKL{\postm}{\priorm} &= \int_\hilb \log\left\{ \frac{d\postm}{d\priorm}\right\} \, \postm(d\ipar)\\ 
    &= -\log \ZZ_0(\obs) - \int_\hilb \Phi(u; \obs) \, \postm(d\ipar). 
\end{aligned}
\end{equation}
Using~\eqref{equ:phi-simp} to expand $\Phi(\ipar; \obs)$, the integral on the right becomes
\[
\int_\hilb \Phi(u; \obs) \, \postm(d\ipar) = 
\frac12 \int_\hilb \ip{\HM \ipar}{\ipar}\, \postm(d\ipar) - \int_\hilb \ip{\fop^*\obs}{\ipar} \postm(d\ipar) 
   + \frac12 |\obs|^2.
\]
The second integral evaluates to $\ip{\fop^*\obs}{\postmean}$, by the definition of the mean of the
measure, and the first integral is evaluated via the formula for the integral of a quadratic
form:
\[
\int_\hilb \ip{\HM \ipar}{\ipar}\, \postm(d\ipar)
= \trace(\HM \postcov) + \ip{\postmean}{\HM\postmean}.
\]
Using the expression for $\ZZ_0$ from Lemma~\ref{lem:Zconst},
\[
\begin{aligned}
-\log \ZZ_0(\obs) &= \frac12 |\obs|^2 - \log \det(I + \HMt)^{-1/2} - \frac12 \ip{\postcov \fop^*\obs}{\fop^*\obs} \\
          &= \frac12 |\obs|^2 + \frac12 \log \det(I + \HMt) - \frac12 \ip{\postmean}{\fop^*\obs},
\end{aligned}
\]
where we have also used the definition of $\postmean$. Substituting into equation~\eqref{equ:dkl}, we obtain
\[
\begin{aligned}
\DKL{\postm}{\priorm} = &\frac12 \log \det(I + \HMt) - \frac12 \ip{\postmean}{\fop^*\obs} \\
&- \frac12 \trace(\HM \postcov) - \frac12 \ip{\postmean}{\HM\postmean} + \ip{\fop^*\obs}{\postmean},
\end{aligned}
\]
which, after some algebraic manipulation and recalling that $\HM = \fop^*\fop$, yields the assertion of the proposition.\qedhere
\end{proof}

Let us note the following interpretation for the last term appearing in $\DKL{\postm}{\priorm}$ 
given in~\eqref{equ:DKLbayes}. Consider the function $\Phi(u) = \frac12 (\fop \ipar - \obs)^T (\fop \ipar - \obs)$,
which is the familiar misfit term in the deterministic interpretation of the corresponding
linear inverse problem. (For notational simplicity 
we have suppressed the dependence of $\Phi$ on the data vector $\obs$.)
Note that 
the variational derivative of $\Phi$ at a point $u \in \hilb$ in direction $h \in \hilb$ is given by,
\[
   \Phi'(u)h = \frac{d}{d\eps}\Big|_{\eps = 0} \Phi(u + \eps h) =  
\eipq{\fop u - \obs}{\fop h}
=\ip{\fop^*(\fop u - \obs)}{h}.
\] 
Next, recall that the mean of the posterior, $\postmean$,
of the present Bayesian linear inverse problem coincides with the MAP
estimator for the inference parameter $u$ and is the global minimizer of
the following regularized cost functional~\cite{Stuart10,DashtiLawStuartEtAl13}
\[
   \J(u) = \Phi(u) + \frac12 \cip{u}{u}
\] 
with minimization done over the space, $\CM = \ran(\priorcov^{1/2}) \subset \hilb$.\footnote{Given a 
Gaussian measure $\mu = \GM{m}{\C}$ on a Hilbert space $H$, 
the space $\ran(\C^{1/2})$ is called the Cameron-Martin space 
corresponding to the measure $\mu$. It is a known result (see e.g.~\cite{DaPrato}) that
if the Hilbert space $H$ is infinite-dimensional, $\mu\big(\ran(\C^{1/2})\big) = 0$.}
The inner-product in the regularization term is given by $\cip{\cdot}{\cdot} = \ip{\priorcov^{-1/2} x}{\priorcov^{-1/2} y}$ for $x, y \in 
\CM$.
We have, by the first order optimality conditions $\J'(\postmean)h = 0$ for every $h \in \CM$, that is,
\[
\ip{\fop^*(\fop u - \obs)}{h} + \cip{\postmean}{h} = 0, \quad \text{ for all } h \in \CM.
\] 
Thus, in particular, $-\ip{\fop^*(\fop u - \obs)}{\postmean} = \cip{\postmean}{\postmean}$.
This leads to the following
alternate form of expression~\eqref{equ:DKLbayes}:
\begin{equation}\label{equ:DKLbayes_alt}
    \DKL{\postm}{\priorm} = \frac12 \left[\log \det(I + \HMt) - \trace(\HM \postcov) +
    \cip{\postmean}{\postmean}\right].
\end{equation}
Note that this expression for the KL divergence $\DKL{\postm}{\priorm}$ 
is the direct extension of the corresponding expression 
in the case of $\hilb = \R^n$ as given in~\eqref{equ:DKL-fd-alt} to infinite dimensions. 

\begin{remark}
A straightforward modification of the arguments leading to equation~\eqref{equ:DKLbayes}, 
for the case of a prior $\priorm = \GM{\priormean}{\priorcov}$, leads to
\[
    \DKL{\postm}{\priorm} = \frac12 \left[\log \det(I + \HMt) - \trace(\HM \postcov) -
    \ip{\postmean - \priormean}{\fop^*(\fop\postmean - \obs)}\right].
\]
Moreover, in view of the argument leading to~\eqref{equ:DKLbayes_alt}, we have:
\[
    \DKL{\postm}{\priorm} = \frac12 \left[\log \det(I + \HMt) - \trace(\HM \postcov) +
    \cip{\postmean - \priormean}{\postmean - \priormean}\right].
\]
\end{remark}

\subsection{Expected information gain}\label{sec:expectedDKL}
Here we derive the expression for the expected information gain. 
We first prove the following technical lemma which is needed in the proof of the main 
result in this section.
\begin{lemma}\label{lem:dblexps}
The following identities hold.
\begin{enumerate}
\item $\dblexp{\ip{\postmean}{\fop^*\obs}} = \trace(\HMt)$
\item $\dblexp{\ip{\postmean}{\HM \postmean}} = \trace(\S\HMt^2)$,
\end{enumerate}
where $\HMt$ and $\S$ be as in~\eqref{equ:HMt} and~\eqref{equ:Sop} respectively.  
\end{lemma}
\begin{proof}
We present the proof of the first statement; the second one follows from a similar 
argument.  Let us begin from the inner expectation. Note that, by the definition of 
$\postmean$ we have,
\[
\ip{\postmean}{\fop^*\obs} = \ip{\postcov\fop^*\obs}{\fop^*\obs}  = \eipq{\obs}{\fop \postcov \fop^* \obs}, 
\]
For clarity let us denote $L = \fop \postcov \fop^*$.
Recall that $\obs | \ipar$ is distributed according to $\GM{\fop u}{\ncov}$, and that we assumed $\ncov = \mat{I}$.
Using the formula for the expectation of a quadratic form (on $\Y = \R^q$), Lemma~\ref{lem:quadforms}, 
we have
\begin{equation*} 
\avemu{\obs | \ipar}{\ip{\postmean}{\fop^*\obs}} = \avemu{\obs | \ipar}{\eipq{\obs}{L \obs}} = \trace(L) + \eipq{\fop\ipar}{L \fop\ipar} = \trace(L) + \ip{\ipar}{\fop^*L\fop\ipar}.   
\end{equation*}
By the comment at the end of Section~\ref{sec:LinearOperators} and
recalling that $\postcov = \Cpv$, we have 
\begin{eqnarray}\label{equ:trL}
\trace(L) \!\!\!&=&\!\!\! \trace(\fop \postcov \fop^*) = \trace(\postcov\HM) = \trace(\priorcov^{1/2}\S\priorcov^{1/2}\HM)\nonumber \\
 &&\hspace{.25in}= \trace(\S\priorcov^{1/2}\HM\priorcov^{1/2}) = \trace(\S \HMt).
\end{eqnarray}
Therefore,
\begin{equation}\label{equ:inner-int}
  \avemu{\obs | \ipar}{\ip{\postmean}{\fop^*\obs}} = \trace(\S \HMt) + \ip{\ipar}{\fop^*L\fop\ipar}.  \end{equation}
Next, to compute the outer expectation we proceed as follows (keep in mind that $\priorm = \GM{0}{\priorcov}$). By Lemma~\ref{lem:quadforms},
\[
\avemu{\priorm}{\ip{\ipar}{\fop^*L\fop\ipar}} = 
\int_\hilb \ip{\ipar}{\fop^*L\fop\ipar} \, \priorm(d\ipar) =
\trace(\fop^*L\fop\priorcov);
\]
and
\begin{eqnarray}\label{equ:trGLGC}
\trace(\fop^*L\fop\priorcov) \!\!\!&=&\!\!\! \trace(\fop^*{\fop \postcov \fop^*}\fop \priorcov) = \trace(\HM \postcov \HM \priorcov)\nonumber \\
 &&\hspace{-.5in}= \trace(\priorcov^{1/2} \HM \postcov \HM \priorcov^{1/2}) = \trace(\priorcov^{1/2} \HM \Cpv \HM \priorcov^{1/2})
 = \trace(\HMt \S \HMt) = \trace(\S \HMt^2).\nonumber
\end{eqnarray}
Thus, combining equations~\eqref{equ:trL}, \eqref{equ:inner-int}, and~\eqref{equ:trGLGC} gives
\begin{eqnarray*}
\dblexp{\ip{\postmean}{\fop^*\obs}} \!\!\!&=&\!\!\! \trace(\S \HMt) + \trace(\S \HMt^2) = \trace(\S\HMt(I + \HMt)) \\
  &&\hspace{-.5in}= \trace(\HMt(I + \HMt)\S)  = \trace(\HMt),
\end{eqnarray*}
which is the first statement of the lemma.\qedhere
\end{proof}

The following theorem is the main result of this section.  
\begin{theorem}
Let $\priorm$ be a centered Gaussian prior measure on $\hilb$, and $\postm = \GM{\postmean}{\postcov}$ be
the posterior measure for a Bayesian linear inverse problem with additive Gaussian noise model
as described in Section~\ref{sec:BayesInverse}.
Then,
\[
    \avemu{\priorm}{\avemu{\obs | \ipar}{\DKL{\postm}{\priorm}}} = 
    \frac12 \log \det(I + \HMt).
\]
\end{theorem}
\begin{proof}
By~\eqref{equ:DKLbayes} we have,
\begin{eqnarray}\label{equ:edkl}
\dblexp{\DKL{\postm}{\priorm}} \!\!\!&=&\!\!\! \frac12 \log \det(I + \HMt) \nonumber \\
  &&\hspace{-1.25in} -\frac12\trace(\HM \postcov) 
  -\frac12 \avemu{\priorm}{\avemu{\obs | \ipar}{\ip{\postmean}{\fop^*(\fop\postmean - \obs)}}}.
\end{eqnarray}
Using the previous lemma we can proceed as follows,
\begin{eqnarray*}
  &&\dblexp{\ip{\postmean}{\fop^*(\fop\postmean - \obs)}} \\
  &&\hspace{.5in}= \dblexp{\ip{\postmean}{\HM \postmean}} - \dblexp{\ip{\postmean}{\fop^*\obs}} \\
  &&\hspace{.5in}=\trace(\S\HMt^2) - \trace(\HMt) \\
  &&\hspace{.5in}=\trace(\S(\HMt - \S^{-1}) \HMt) =-\trace(\S\HMt). 
\end{eqnarray*}
Thus, since $\trace(\HM \postcov)\!=\!\trace(\HMt\S)\! = \!\trace(\S \HMt)$, 
the expression for the expected information gain in~\eqref{equ:edkl} simplifies to
$\dblexp{\DKL{\postm}{\priorm}} \!=\!\frac12 \log \det(I + \HMt)$.\qedhere
\end{proof}

The result above provides the infinite-dimensional analogue of Bayesian D-optimality. 
As mentioned in the introduction, an experimental design $\design$ enters the Bayesian inverse problem through the 
data likelihood. This dependence to $\design$, in the present Gaussian linear case, is manifested 
through a $\design$ dependent misfit 
Hessian, $\HM = \HM(\design)$. Consequently, the D-optimal 
design problem in the infinite-dimensional Hilbert space setting is given by,
\[
\mathop{\mathrm{maximize}~}_{\design \in \Xi} \log \det(I + \HMt(\design)),
\]
where $\Xi$ is the design space which needs to be specified in a given experimental design problem. 
\begin{remark}
As mentioned earlier, in a large class of Bayesian inverse problems, $\HMt$ admits a low-rank 
approximation, 
\[
\HMt v \approx \sum_{i = 1}^r \lambda_i \ip{e_i}{v}e_i, \quad v \in \hilb, 
\]
where $r$ is the numerical rank 
of $\HMt$ and $\{\lambda_i\}_{i=1}^r$ are the dominant eigenvalues of $\HMt$ with respective 
eigenvectors $\{e_i\}_{i=1}^r$.
Thus, 
one can use the following approximation  
\[
\log \det(I + \HMt) \approx \sum_{i = 1}^r \log (1 + \lambda_i),
\]
which enables an efficient means of approximating the expected information gain.
\end{remark}

\section{Expected mean square error of the MAP estimator and Bayesian A-optimality}\label{sec:bayesrisk}
In this section, we consider another well known optimal experimental design
criterion, Bayesian A-optimality, which aims to minimize the
trace of the posterior covariance operator. 
It is well known in the statistics literature that for inference problems
with a finite-dimensional parameter, this is equivalent to minimizing the expected mean square
error of the mean posterior which, in the
case of a Bayesian linear inverse problem, coincides with the MAP estimator.
In this section, we extend this result to the infinite-dimensional Hilbert space setting.

The MSE of the MAP estimator $\postmean$ is
\[
\MSE(\postmean; \ipar) = \avey{\norm{\ipar - \postmean}^2}.
\]
The $\MSE$ is also referred to as the \emph{risk} of the estimator $\postmean$, corresponding
to a quadratic loss function.
A straightforward calculation shows that
\begin{equation}\label{equ:MSE}
    \MSE(\postmean; \ipar) = \norm{\ipar - \avey{\postmean}}^2 + \avey{ \norm{\postmean - \avey{\postmean}}^2},
\end{equation}
Note that the first term in~\eqref{equ:MSE} quantifies the magnitude of estimation bias,
and the second term describes the variability of the estimator around its mean.
The following technical Lemma provides the expression for $\MSE(\postmean; \ipar)$
in the infinite-dimensional Hilbert space setting.

\begin{lemma}\label{lem:MSE}
Let $\postmean$ be the MAP estimator for $\ipar$ as in~\eqref{equ:postmeas}. Then,
\[
    \MSE(\postmean; \ipar) = \norm{(\postcov\HM - I)u}^2 + \trace(\postcov^2 \HM).
\]
\end{lemma}
\begin{proof}
Consider the expression for $\MSE(\postmean; \ipar)$ given in~\eqref{equ:MSE}.
For the first term in the sum, we have
\[
\ipar - \avey{\postmean} = \ipar - \avey{\postcov\fop^* \obs} =
\ipar - \postcov \fop^*\fop u = (I - \postcov\HM)u.
\]
Next, note that $\xi(\obs) = \postmean - \avey{\postmean}$ has law $\mu = \GM{0}{\Q}$
with $\Q = (\postcov\fop^*)(\postcov\fop^*)^* = \postcov\HM\postcov$. Therefore,
\[
\avey{ \norm{\postmean - \avey{\postmean}}^2} =
\int_\hilb \norm{\xi}^2 \, \mu(d\xi) = \trace(\postcov\HM\postcov) = \trace(\postcov^2\HM).\qedhere
\]
\end{proof}

Next, we consider the average over the prior measure of the $\MSE$,
\begin{equation*}
   \avemu{\priorm}{\MSE(\postmean; \ipar)} 
                 = \int_\hilb \int_\mathscr{Y} \norm{\ipar - \postmean}^2  \like( \vec{y} | \ipar) \, d\vec{y} \, \priorm(d\ipar),
\end{equation*}
which is also known as the Bayes risk of the estimator $\postmean$, corresponding to a quadratic loss function~\cite{carlin:1997,berger:1985}.
The following result extends the well known result regarding the connection between 
the Bayes risk of the MAP estimator and the trace of the posterior covariance, for a Bayesian linear
inverse problem, to the infinite-dimensional Hilbert space setting.
\begin{theorem}
Let $\priorm$ be a centered Gaussian prior measure 
on $\hilb$, and $\postm = \GM{\postmean}{\postcov}$ be 
the posterior measure for a Bayesian linear inverse problem with additive Gaussian noise model 
as described in Section~\ref{sec:BayesInverse}. Then,
$\avemu{\priorm}{\MSE(\postmean; \ipar)} = \trace(\postcov)$.
\end{theorem}
\begin{proof}
By Lemma~\ref{lem:MSE},
\begin{equation}\label{equ:risk}
\avemu{\priorm}{\MSE(\postmean; \ipar)} = \int_\hilb \norm{(\postcov\HM - I)u}^2 \, \priorm(du) + \trace(\postcov^2 \HM),
\end{equation}
and since $(\postcov\HM - I)u \sim \GM{0}{(\postcov\HM - I)\priorcov(\postcov\HM - I)^*} =: \mu$,
\[
\int_\hilb \norm{(\postcov\HM - I)u}^2 \, \priorm(du) =
\int_\hilb \norm{\xi}^2 \, \mu(d\xi) = \trace((\postcov\HM - I)\priorcov(\postcov\HM - I)^*).
\]
We proceed as follows,
\begin{eqnarray*}
\trace((\postcov\HM - I)\priorcov(\postcov\HM - I)^*) \!\!\!&=&\!\!\!
\trace((\postcov\HM - I)^*(\postcov\HM - I)\priorcov)\\
&&\hspace{-1in}= \trace(\HM\postcov^2\HM\priorcov) - \trace(\HM\postcov\priorcov)
- \trace(\postcov\HM\priorcov) + \trace(\priorcov).
\end{eqnarray*}
Let us first consider the last two terms; 
recalling that $\postcov = \prhalf \S \prhalf$, we have
\[
\begin{aligned}
- \trace(\postcov\HM\priorcov) + \trace(\priorcov)
&= -\trace(\prhalf\S\prhalf\HM\priorcov) + \trace(\priorcov)\\
&= -\trace(\priorcov\S\HMt) + \trace(\priorcov)\\
&= \trace(\priorcov\S (\S^{-1} - \HMt)) = \trace(\priorcov\S) = \trace(\postcov).
\end{aligned}
\]
Thus, by~\eqref{equ:risk}, we have,
\[
\avemu{\priorm}{\MSE(\postmean; \ipar)} =
\trace(\postcov^2 \HM) + \trace(\HM\postcov^2\HM\priorcov) - \trace(\HM\postcov\priorcov) 
+ \trace(\postcov).
\]
Hence, showing that the first three three sum to zero completes the proof. To this end, we note
that $\trace(\HM\postcov\priorcov) = \trace(\HMt \S \priorcov)$ and that
\begin{eqnarray*}
\trace(\postcov^2 \HM) + \trace(\HM\postcov^2\HM\priorcov) \!\!\!&=&\!\!\! \trace(\HMt\S \priorcov\S) + \trace(\HMt \S \priorcov\S \HMt)\\ 
&=&\!\!\! \trace\big(\HMt\S\priorcov\S (I + \HMt)\big)
= \trace(\HMt\S\priorcov).\qedhere
\end{eqnarray*}
\end{proof}

\appendix
\section{Proof of Lemma~\ref{lem:quadforms}}\label{appdx:quadforms}
Let $\{e_i\}_1^\infty$ be a complete orthonormal set in $\hilb$, 
and denote by $\Pi_n$ the orthogonal projection
of $\hilb$ onto $\span\{e_1, \ldots, e_n\}$; that is, for $x \in \hilb$, 
$\Pi_n(x) = \sum_{i = 1}^n \ip{e_i}{x} e_i$.
First note that, 
\begin{equation*}
\begin{aligned}
   \int_{\hilb} \ip{\A x}{x} =  &\int_{\hilb} \ip{\A (x-m)}{x-m} \, \mu(dx)\\ 
                             &+ \int_{\hilb} \ip{\A x}{m} \, \mu(dx)  
                             + \int_{\hilb} \ip{\A m}{x} \, \mu(dx) 
                             - \ip{\A m}{m}. 
\end{aligned}
\end{equation*}
Now by the definition of the mean of the measure, the last three terms sum to $\ip{\A m}{m}$.
Thus, the rest of the proof consists of showing $\int_{\hilb} \ip{\A (x-m)}{x-m} \, \mu(dx) = \trace(\A \Q)$.
Note that for every $x \in \hilb$, $x - m = \lim_{n \to \infty} \Pi_n(x - m)$, and thus,
\[
    \lim_{n \to \infty} \ip{\A \Pi_n(x - m)}{\Pi_n(x - m)} = \ip{\A (x-m)}{x-m}.
\]
Moreover, we note that, $|\ip{\A \Pi_n(x - m)}{\Pi_n(x - m)} | \leq \|A\| \norm{x - m}^2$, and
since $\Q$ is trace-class, the measure $\mu$ has a bounded second moment; hence, $\int_\hilb \norm{x - m}^2\, \mu(dx) < \infty$.
Therefore, we can apply Lebesgue-Dominated Convergence Theorem to get,
\begin{equation}\label{equ:dct}
 \! \!\! \lim_{n \to \infty} \int_\hilb \!\ip{\A \Pi_n(x-m)}{\Pi_n(x-m)} \, \mu(dx) \!= \!
\int_\hilb \!\ip{\A(x-m)}{x-m} \, \mu(dx).\!
\end{equation}
Next, let $\{e_i\}_1^\infty$ be the (complete) set of eigenvectors of $\Q$ with corresponding
(real) eigenvalues $\{\lambda_i\}_1^\infty$. We know that $\A \Q$ is trace-class with,
\begin{equation}\label{equ:trAQ}
    \trace(\A \Q) = \sum_i \ip{\A \Q e_i}{e_i} = \sum_i \lambda_i \ip{\A e_i}{e_i}.
\end{equation}
Also, note,
\begin{equation*}
\begin{aligned}
   \int_\hilb &\ip{\A \Pi_n(x-m)}{\Pi_n(x-m)}\, \mu(dx)\\ 
   &= 
   \sum_{i,j=1}^n \int_\hilb \ip{\A e_i}{e_j} \ip{x-m}{e_i}\ip{x-m}{e_j} \mu(dx)\\ 
   &= 
   \sum_{i,j=1}^n \ip{\A e_i}{e_j} \int_\hilb \ip{x-m}{e_i}\ip{x-m}{e_j} \mu(dx)\\ 
   &=
   \sum_{i,j=1}^n \ip{\A e_i}{e_j} \ip{\Q e_i}{e_j}
   = 
   \sum_{i=1}^n \lambda_i \ip{\A e_i}{e_i}.
\end{aligned}
\end{equation*}
Therefore, combining this last result with~\eqref{equ:dct} and~\eqref{equ:trAQ}, we get
\[
\begin{aligned}
   \int_\hilb \ip{\A (x-m)}{x-m} \, \mu(dx) &= \lim_{n \to \infty} \int_\hilb \ip{\A \Pi_n(x-m)}{\Pi_n(x-m)} \, \mu(dx)\\
                                   &= \lim_{n \to \infty} \sum_{i=1}^n \lambda_i \ip{\A e_i}{e_i} = \trace(\A \Q).~\square
\end{aligned}
\]

\bibliographystyle{abbrv}

\begin{thebibliography}{10}

\bibitem{AlexanderianEtAl2014}
A.~Alexanderian, N.~Petra, G.~Stadler, and O.~Ghattas.
\newblock A-optimal design of experiments for infinite-dimensional {B}ayesian
  linear inverse problems with regularized $\ell_0$-sparsification.
\newblock {\em SIAM Journal on Scientific Computing}, 2014.
\newblock to appear.

\bibitem{AtkinsonDonev92}
A.~C. Atkinson and A.~N. Donev.
\newblock {\em Optimum Experimental Designs}.
\newblock Oxford, 1992.

\bibitem{berger:1985}
J.~O. Berger.
\newblock {\em Statistical decision theory and {B}ayesian analysis}.
\newblock Springer, 1985.

\bibitem{BuiEtAl2013}
T.~Bui-Thanh, O.~Ghattas, J.~Martin, and G.~Stadler.
\newblock A computational framework for infinite-dimensional {B}ayesian inverse
  problems part i: The linearized case, with application to global seismic
  inversion.
\newblock {\em SIAM Journal on Scientific Computing}, 35(6):A2494--A2523, 2013.

\bibitem{carlin:1997}
B.~P. Carlin and T.~A. Louis.
\newblock Bayes and empirical bayes methods for data analysis.
\newblock {\em Statistics and Computing}, 7(2):153--154, 1997.

\bibitem{ChalonerVerdinelli95}
K.~Chaloner and I.~Verdinelli.
\newblock Bayesian experimental design: {A} review.
\newblock {\em Statistical Science}, 10(3):273--304, 1995.

\bibitem{Conway}
J.~B. Conway.
\newblock {\em A course in operator theory}.
\newblock American Mathematical Soc., 2000.

\bibitem{DaPrato}
G.~Da~Prato.
\newblock {\em An introduction to infinite-dimensional analysis}.
\newblock Springer, 2006.

\bibitem{DaPratoZabczyk}
G.~Da~Prato and J.~Zabczyk.
\newblock {\em Second order partial differential equations in Hilbert spaces},
  volume 293.
\newblock Cambridge University Press, 2002.

\bibitem{DashtiLawStuartEtAl13}
M.~Dashti, K.~J. Law, A.~M. Stuart, and J.~Voss.
\newblock {MAP} estimators and their consistency in {B}ayesian nonparametric
  inverse problems.
\newblock {\em Inverse Problems}, 29, 2013.

\bibitem{FlathEtAl2011}
H.~P. Flath, L.~C. Wilcox, V.~Ak{\c{c}}elik, J.~Hill, B.~van Bloemen~Waanders,
  and O.~Ghattas.
\newblock Fast algorithms for {B}ayesian uncertainty quantification in
  large-scale linear inverse problems based on low-rank partial hessian
  approximations.
\newblock {\em SIAM Journal on Scientific Computing}, 33(1), 2011.

\bibitem{Kullback1951}
S.~Kullback and R.~A. Leibler.
\newblock On information and sufficiency.
\newblock {\em The Annals of Mathematical Statistics}, 22(1):79--86, 03 1951.

\bibitem{Pukelsheim06}
F.~Pukelsheim.
\newblock {\em Optimal design of experiments}, volume~50.
\newblock siam, 2006.

\bibitem{ReedSimon}
M.~Reed and B.~Simon.
\newblock {\em Methods of Modern Mathematical Physics: Vol.: 1.: Functional
  Analysis}.
\newblock Academic press, 1972.

\bibitem{Simon1977}
B.~Simon.
\newblock Notes on infinite determinants of {H}ilbert space operators.
\newblock {\em Advances in Mathematics}, 24:244--273, 1977.

\bibitem{Stuart10}
A.~M. Stuart.
\newblock Inverse problems: a {B}ayesian perspective.
\newblock {\em Acta Numerica}, 19:451--559, 2010.

\bibitem{Ucinski05}
D.~Uci{\'n}ski.
\newblock {\em Optimal measurement methods for distributed parameter system
  identification}.
\newblock CRC Press, Boca Raton, 2005.

\end{thebibliography}

\end{document}